\documentclass[11pt,twoside]{article}
\usepackage{amsmath, amsthm, amscd, amsfonts, amssymb, graphicx, color}

\date{}

\begin{document}

\centerline{}

\centerline {\Large{\bf  Generalized fusion frame in tensor product of}}

\centerline{\Large{\bf Hilbert spaces}}

%% My definition
\newcommand{\mvec}[1]{\mbox{\bfseries\itshape #1}}
\centerline{}
\centerline{\textbf{Prasenjit Ghosh}}
\centerline{Department of Pure Mathematics, University of Calcutta,}
\centerline{35, Ballygunge Circular Road, Kolkata, 700019, West Bengal, India}
\centerline{e-mail: prasenjitpuremath@gmail.com}
\centerline{}
\centerline{\textbf{T. K. Samanta}}
\centerline{Department of Mathematics, Uluberia College,}
\centerline{Uluberia, Howrah, 711315,  West Bengal, India}
\centerline{e-mail: mumpu$_{-}$tapas5@yahoo.co.in}

\newtheorem{Theorem}{\quad Theorem}[section]

\newtheorem{definition}[Theorem]{\quad Definition}

\newtheorem{theorem}[Theorem]{\quad Theorem}

\newtheorem{remark}[Theorem]{\quad Remark}

\newtheorem{corollary}[Theorem]{\quad Corollary}

\newtheorem{note}[Theorem]{\quad Note}

\newtheorem{lemma}[Theorem]{\quad Lemma}

\newtheorem{example}[Theorem]{\quad Example}

\newtheorem{result}[Theorem]{\quad Result}
\newtheorem{conclusion}[Theorem]{\quad Conclusion}

\newtheorem{proposition}[Theorem]{\quad Proposition}

\begin{abstract}
\textbf{\emph{Generalized fusion frame and some of their properties in tensor product of Hilbert spaces are described.\,Also, the canonical dual g-fusion frame in tensor product of Hilbert spaces is considered.\,Finally, the frame operator for a pair of $g$-fusion Bessel sequences in tensor product of Hilbert spaces is presented.}}
\end{abstract}
{\bf Keywords:}  \emph{Frame, fusion frame, $g$-frame, $g$-fusion frame, frame operator,\\ \smallskip\hspace{2.2 cm}Tensor product of Hilbert spaces, Tensor product of frames.}


%=====================================
\section{Introduction}
%=====================================

Frame for Hilbert space was first introduced by Duffin and Schaeffer \cite{Duffin} in 1952 to study some fundamental problems in non-harmonic Fourier series.\;The formal definition of frame in the abstract Hilbert spaces was given by Daubechies et al.\,\cite{Daubechies} in 1986.\;Frame theory has been widely used in signal and image processing, filter bank theory, coding and communications, system modeling and so on.\;Several generalizations of frames  namely, \,$K$-frames, \,$g$-frames, fusion frames, \,$g$-fusion frames etc.\;had been introduced in recent times.\;Sun \cite{Sun}\, introduced \,$g$-frame and \,$g$-Riesz basis in complex Hilbert spaces and discussed several properties of them.\;P.\,Casazza and G.\,Kutyniok \cite{Kutyniok} were first to introduced the notion of fusion frames.\;P.\,Gavruta \cite{Gavruta} discussed the duality of fusion frames and defined the frame operator for pair of fusion Bessel sequences in Hilbert spaces.\;Generalized fusion frames in Hilbert spaces were presented by Sadri et al.\,\cite{Ahmadi} to generalize the theory of fusion frame and \,$g$-frame.\,Generalized atomic subspaces for operators in Hilbert spaces were studied by P.\,Ghosh and T.\,K.\,Samanta {\cite{Ghosh}} and they were also presented the stability of dual \,$g$-fusion frames in Hilbert spaces in {\cite{P}}.

The basic concepts of tensor product of Hilbert spaces were presented by S.\,Rabin-son \cite{S}.\;Frames and Bases in Tensor Product of Hilbert spaces were introduced by A.\,Khosravi and M.\,S.\,Asgari \cite{A}.\;Reddy et al.\,\cite{Upender} also studied the frame in tensor product of Hilbert spaces and presented the tensor frame operator on tensor product of Hilbert spaces.\;The concepts of fusion frames and \,$g$-frames in tensor product of Hilbert spaces were introduced by Amir Khosravi and M.\,Mirzaee Azandaryani \cite{Mir}.  

In this paper, generalized fusion frame or \,$g$-fusion frame in tensor product of Hilbert spaces is presented and some of their properties are going to be established.\,The canonical dual\;$g$-fusion frame in tensor product of Hilbert spaces is also discussed.\;At the end, the relation between the frame operators for the pair of \,$g$-fusion Bessel sequences in Hilbert spaces and their tensor product are obtained.

Throughout this paper,\;$H \;\text{and}\; K$\, are considered to be separable Hilbert spaces with associated inner products \,$\left <\,\cdot \,,\, \cdot\,\right>_{1} \;\text{and}\; \left <\,\cdot \,,\, \cdot\,\right>_{2}$\,.\,$I_{H}$\, and \,$I_{K}$\, denotes the identity operators on \,$H$\, and \,$K$, respectively.\;$\mathcal{B}\,(\,H \,,\, K\,)$\; is the collection of all bounded linear operators from \,$H \;\text{to}\; K$.\;In particular \,$\mathcal{B}\,(\,H\,)$\; denote the space of all bounded linear operators on \,$H$.\,$P_{\,V}$\, denote the orthogonal projection onto the closed subspace \,$V \,\subset\, H$.\;$\left\{\,V_{i}\,\right\}_{ i \,\in\, I}$\, and \,$\left\{\,W_{j}\,\right\}_{ j \,\in\, J}$\, are the collections of closed subspaces of \,$H \;\text{and}\; K$, where \,$I,\, J$\, are subsets of  integers \,$\mathbb{Z}$.\;$\left\{\,H_{i}\,\right\}_{ i \,\in\, I}$\, and \,$\left\{\,K_{j}\,\right\}_{ j \,\in\, J}$\, are the collections of Hilbert spaces.\,$\left\{\,\Lambda_{i} \,\in\, \mathcal{B}\,(\,H \,,\, H_{i}\,)\,\right\}_{i \,\in\, I}$\, and \,$\left\{\,\Gamma_{j} \,\in\, \mathcal{B}\,(\,K \,,\, K_{j}\,)\,\right\}_{j \,\in\, J}$\, denotes the sequences of operators.\;Define the space
\[l^{\,2}\left(\,\left\{\,H_{i}\,\right\}_{ i \,\in\, I}\,\right) \,=\, \left \{\,\{\,f_{\,i}\,\}_{i \,\in\, I} \,:\, f_{\,i} \;\in\; H_{i},\; \sum\limits_{\,i \,\in\, I}\, \left \|\,f_{\,i}\,\right \|^{\,2} \,<\, \infty \,\right\}\]
with inner product is given by \,$\left<\,\{\,f_{\,i}\,\}_{i \,\in\, I} \,,\, \{\,g_{\,i}\,\}_{ i \,\in\, I}\,\right> \;=\; \sum\limits_{\,i \,\in\, I}\, \left<\,f_{\,i} \,,\, g_{\,i}\,\right>_{H_{i}}$.\;Clearly \,$l^{\,2}\left(\,\left\{\,H_{i}\,\right\}_{ i \,\in\, I}\,\right)$\; is a Hilbert space with respect to the above inner product \cite{Ahmadi}.\,Similarly, we can define the space \,$l^{\,2}\left(\,\left\{\,K_{j}\,\right\}_{ j \,\in\, J}\,\right)$.

%=====================================
\section{Preliminaries}
%=====================================

\smallskip\hspace{.6 cm}

\begin{theorem}\cite{Gavruta}\label{th0.001}
Let \,$V \,\subset\, H$\; be a closed subspace and \,$T \,\in\, \mathcal{B}\,(\,H\,)$.\;Then \,$P_{\,V}\, T^{\,\ast} \,=\, P_{\,V}\,T^{\,\ast}\, P_{\,\overline{T\,V}}$.\;If \,$T$\; is an unitary operator (\,i\,.\,e., \,$T^{\,\ast}\, T \,=\, I_{H}$\,), then \,$P_{\,\overline{T\,V}}\;T \,=\, T\,P_{\,V}$.
\end{theorem}

\begin{theorem}\cite{Jain}\label{th0.01}
The set \,$\mathcal{S}\,(\,H\,)$\; of all self-adjoint operators on \,$H$\; is a partially ordered set with respect to the partial order \,$\leq$\, which is defined as for \,$T,\,S \,\in\, \mathcal{S}\,(\,H\,)$ 
\[T \,\leq\, S \,\Leftrightarrow\, \left<\,T\,f \,,\, f\,\right> \,\leq\, \left<\,S\,f \,,\, f\,\right>\; \;\forall\; f \,\in\, H.\] 
\end{theorem}

\begin{definition}\cite{O}
A sequence \,$\left\{\,f_{\,i}\,\right\}_{i \,\in\, I}$\, of elements in \,$H$\, is a frame for \,$H$\, if there exist constants \,$A,\, B \,>\, 0$\, such that
\[ A\, \|\,f\,\|^{\,2} \,\leq\, \sum\limits_{i \,\in\, I}\, \left|\, \left<\,f \,,\, f_{\,i} \, \right>\,\right|^{\,2} \,\leq\, B \,\|\,f\,\|^{\,2}\; \;\forall\; f \,\in\, H. \]
The constants \,$A$\, and \,$B$\, are called frame bounds.
\end{definition}

\begin{definition}\cite{Kutyniok}
Let \,$\left\{\,v_{i}\,\right\}_{ i \,\in\, I}$\, be a collection of positive weights i\,.\,e., \,$v_{\,i} \,>\, 0\, \;\forall\; i \,\in\, I$.\;A family of weighted closed subspaces \,$\left\{\, (\,V_{i},\, v_{i}\,) \,:\, i \,\in\, I\,\right\}$\, is called a fusion frame for \,$H$\, if there exist constants \,$0 \,<\, A \,\leq\, B \,<\, \infty$\, such that
\[A \;\left\|\,f \,\right\|^{\,2} \,\leq\, \sum\limits_{\,i \,\in\, I}\, v_{i}^{\,2} \;\left\|\, P_{\,V_{i}}\,(\,f\,) \,\right\|^{\,2} \,\leq\, B \, \left\|\, f \, \right\|^{\,2}\; \;\forall\; f \,\in\, H.\]
The constants \,$A,\, B$\; are called fusion frame bounds.
\end{definition}

\begin{definition}\cite{Sun}
A sequence \,$\left\{\,\Lambda_{i} \,\in\, \mathcal{B}\,(\,H,\, H_{i}\,) \,:\, i \,\in\, I\,\right\}$\, is called a generalized frame or g-frame for \,$H$\; with respect to \,$\left\{\,H_{i}\,\right\}_{i \,\in\, I}$\; if there exist two positive constants \,$A$\; and \,$B$\; such that
\[A \;\left \|\, f \,\right \|^{\,2} \,\leq\, \sum\limits_{\,i \,\in\, I}\, \left\|\,\Lambda_{i}\,(\,f\,) \,\right\|^{\,2} \,\leq\, B \; \left\|\, f \, \right\|^{\,2}\; \;\forall\; f \,\in\, H.\]
\,$A$\; and \,$B$\; are called the lower and upper bounds of g-frame, respectively.
\end{definition}

\begin{definition}\cite{Ahmadi}
Let \,$\left\{\,v_{i}\,\right\}_{ i \,\in\, I}$\, be a collection of positive weights.\;Then the family \,$\Lambda \,=\, \{\,\left(\,V_{i},\, \Lambda_{i},\, v_{i}\,\right)\,\}_{i \,\in\, I}$\; is called a generalized fusion frame or a g-fusion frame for \,$H$\; respect to \,$\left\{\,H_{i}\,\right\}_{i \,\in\, I}$\; if there exist constants \,$0 \,<\, A \,\leq\, B \,<\, \infty$\; such that
\begin{equation}\label{eq1.01}
A \;\left \|\,f \,\right \|^{\,2} \,\leq\, \sum\limits_{\,i \,\in\, I}\,v_{i}^{\,2}\, \left\|\,\Lambda_{i}\,P_{\,V_{i}}\,(\,f\,) \,\right\|^{\,2} \,\leq\, B \; \left\|\, f \, \right\|^{\,2}\; \;\forall\; f \,\in\, H.
\end{equation}
The constants \,$A$\; and \,$B$\; are called the lower and upper bounds of g-fusion frame, respectively.\,If \,$A \,=\, B$\; then \,$\Lambda$\; is called tight g-fusion frame and if \;$A \,=\, B \,=\, 1$\, then we say \,$\Lambda$\; is a Parseval g-fusion frame.\;If  \,$\Lambda$\; satisfies the inequality
\[\sum\limits_{\,i \,\in\, I}\,v_{i}^{\,2}\, \left\|\,\Lambda_{i}\,P_{\,V_{i}}\,(\,f\,) \,\right\|^{\,2} \,\leq\, B \; \left\|\, f \, \right\|^{\,2}\; \;\forall\; f \,\in\, H\]
then it is called a g-fusion Bessel sequence with bound \,$B$\; in \,$H$. 
\end{definition}

\begin{definition}\cite{Ahmadi}\label{defn1}
Let \,$\Lambda \,=\, \{\,\left(\,V_{i},\, \Lambda_{i},\, v_{i}\,\right)\,\}_{i \,\in\, I}$\, be a g-fusion Bessel sequence in \,$H$\, with a bound \,$B$.\;The synthesis operator \,$T_{\Lambda}$\, of \,$\Lambda$\; is defined as 
\[ T_{\Lambda} \,:\, l^{\,2}\left(\,\left\{\,H_{i}\,\right\}_{ i \,\in\, I}\,\right) \,\to\, H,\]
\[T_{\Lambda}\,\left(\,\left\{\,f_{\,i}\,\right\}_{i \,\in\, I}\,\right) \,=\,  \sum\limits_{\,i \,\in\,I}\, v_{i}\, P_{\,V_{i}}\,\Lambda_{i}^{\,\ast}\,f_{i}\; \;\;\forall\; \{\,f_{i}\,\}_{i \,\in\, I} \,\in\, l^{\,2}\left(\,\left\{\,H_{i}\,\right\}_{ i \,\in\, I}\,\right)\] and the analysis operator is given by 
\[ T_{\Lambda}^{\,\ast} \,:\, H \,\to\, l^{\,2}\left(\,\left\{\,H_{i}\,\right\}_{ i \,\in\, I}\,\right),\; T_{\Lambda}^{\,\ast}\,(\,f\,) \,=\,  \left\{\,v_{i}\,\Lambda_{i}\, P_{\,V_{i}}\,(\,f\,)\,\right\}_{ i \,\in\, I}\; \;\forall\; f \,\in\, H.\]
The g-fusion frame operator \,$S_{\Lambda} \,:\, H \,\to\, H$\; is defined as follows:
\[S_{\Lambda}\,(\,f\,) \,=\, T_{\Lambda}\,T_{\Lambda}^{\,\ast}\,(\,f\,) \,=\, \sum\limits_{\,i \,\in\, I}\, v_{i}^{\,2}\; P_{\,V_{i}}\, \Lambda_{i}^{\,\ast}\; \Lambda_{i}\, P_{\,V_{i}}\,(\,f\,)\; \;\forall\; f \,\in\, H.\]
\end{definition}

\begin{note}\cite{Ahmadi}
Let \,$\Lambda \,=\, \{\,\left(\,V_{i},\, \Lambda_{i},\, v_{i}\,\right)\,\}_{i \,\in\, I}$\, be a g-fusion Bessel sequence in \,$H$.\,Then it can be easily verify that for all \,$f \,\in\, H$ 
\[\left<\,S_{\Lambda}\,(\,f\,) \,,\, f\,\right> \,=\, \sum\limits_{\,i \,\in\, I}\, v_{i}^{\,2}\, \left\|\,\Lambda_{i}\, P_{\,V_{i}}\,(\,f\,) \,\right\|^{\,2} \,=\, \left\|\,T_{\Lambda}^{\,\ast}\,(\,f\,)\,\right\|^{\,2}.\]
If \,$\Lambda$\, is a g-fusion frame with bounds \,$A$\, and \,$B$\, then from (\ref{eq1.01}),
\[\left<\,A\,f \,,\, f\,\right> \,\leq\, \left<\,S_{\Lambda}\,(\,f\,) \,,\, f\,\right> \,\leq\, \left<\,B\,f \,,\, f\,\right>\; \;\forall\; f \,\in\, H.\]
Now, according to the Theorem (\ref{th0.01}), we can write, \,$A\,I_{H} \,\leq\,S_{\Lambda} \,\leq\, B\,I_{H}$.\;The operator \,$S_{\Lambda}$\; is bounded, self-adjoint, positive and invertible.\;Also, \,$ B^{\,-1}\,I_{H} \,\leq\, S_{\,\Lambda}^{\,-1} \,\leq\, A^{\,-1}\,I_{H}$.\;Hence, reconstruction formula for any \,$f \,\in\, H$, is given by  
\[ f \,=\, \sum\limits_{\,i \,\in\, I}\,v^{\,2}_{\,i}\,P_{\,V_{i}}\,\Lambda^{\,\ast}_{i}\,\Lambda_{i}\,P_{\,V_{i}}\,S^{\,-\, 1}_{\Lambda}\,(\,f\,) \,=\, \sum\limits_{\,i \,\in\, I}\,v^{\,2}_{\,i}\,S^{\,-\, 1}_{\Lambda}\,P_{\,V_{i}}\,\Lambda^{\,\ast}_{i}\,\Lambda_{i}\,P_{\,V_{i}}\,(\,f\,).\]
\end{note}

\begin{definition}\cite{Ahmadi}\label{defn2}
Let \,$\Lambda \,=\, \{\,\left(\,V_{i},\, \Lambda_{i},\, v_{i}\,\right)\,\}_{i \,\in\, I}$\, be a g-fusion frame for \,$H$\, with frame operator \,$S_{\Lambda}$.\;Then the g-fusion frame \,$\left\{\,\left(\,S^{\,-\, 1}_{\Lambda}\,V_{i},\, \Lambda_{i}\,P_{\,V_{i}}\,S^{\,-\, 1}_{\Lambda},\, v_{i}\,\right)\,\right\}_{i \,\in\, I}$\, is called the canonical dual g-fusion frame of \,$\Lambda$. 
\end{definition}

\begin{definition}\cite{Ghosh}
Let \,$\Lambda \,=\, \left\{\,\left(\,V_{i},\, \Lambda_{i},\, v_{i}\,\right)\,\right\}_{i \,\in\, I}$\, and \,$\Lambda^{\,\prime} \,=\, \left\{\,\left(\,V^{\,\prime}_{i},\, \Lambda^{\,\prime}_{i},\, v^{\,\prime}_{i}\,\right)\,\right\}_{i \,\in\, I}$\, be two g-fusion Bessel sequences in \,$H$\; with bounds \,$D_{\,1}$\, and \,$D_{\,2}$, respectively.\;Then the operator \,$S_{\Lambda\,\Lambda^{\,\prime}} \,:\, H \,\to\, H$, defined by
\[S_{\Lambda\,\Lambda^{\,\prime}}\,(\,f\,) \;=\; \sum\limits_{\,i \,\in\, I}\,v_{\,i}\,v^{\,\prime}_{\,i}\,P_{\,V_{i}}\,\Lambda_{i}^{\,\ast}\,\Lambda_{i}^{\,\prime} \,P_{\,V^{\,\prime}_{i}}\,(\,f\,)\;\; \;\forall\; f \,\in\, H,\] is called the frame operator for the pair of g-fusion Bessel sequences \,$\Lambda$\, and \,$\Lambda^{\,\prime}$.
\end{definition}

There are several ways to introduced the tensor product of Hilbert spaces.\;The tensor product of Hilbert spaces \,$H$\, and \,$K$\, is a certain linear space of operators which was represented by Folland in \cite{Folland}, Kadison and Ringrose in \cite{Kadison}.\\

\begin{definition}\cite{Upender}
The tensor product of Hilbert spaces \,$H$\, and \,$K$\, is denoted by \,$H \,\otimes\, K$\, and it is defined to be an inner product space associated with the inner product
\begin{equation}\label{eq1.001}   
\left<\,f \,\otimes\, g \,,\, f^{\,\prime} \,\otimes\, g^{\,\prime}\,\right> \,=\, \left<\,f  \,,\, f^{\,\prime}\,\right>_{\,1}\;\left<\,g  \,,\, g^{\,\prime}\,\right>_{\,2}\; \;\forall\; f,\, f^{\,\prime} \,\in\, H\; \;\text{and}\; \;g,\, g^{\,\prime} \,\in\, K.
\end{equation}
The norm on \,$H \,\otimes\, K$\, is given by 
\begin{equation}\label{eq1.0001}
\left\|\,f \,\otimes\, g\,\right\| \,=\, \|\,f\,\|_{\,1}\;\|\,g\,\|_{\,2}\; \;\forall\; f \,\in\, H\; \;\&\; \,g \,\in\, K.
\end{equation}
The space \,$H \,\otimes\, K$\, is complete with respect to the above inner product.\;Therefore the space \,$H \,\otimes\, K$\, is a Hilbert space.     
\end{definition} 

For \,$Q \,\in\, \mathcal{B}\,(\,H\,)$\, and \,$T \,\in\, \mathcal{B}\,(\,K\,)$, the tensor product of operators \,$Q$\, and \,$T$\, is denoted by \,$Q \,\otimes\, T$\, and defined as 
\[\left(\,Q \,\otimes\, T\,\right)\,A \,=\, Q\,A\,T^{\,\ast}\; \;\forall\; \;A \,\in\, H \,\otimes\, K.\]
It can be easily verified that \,$Q \,\otimes\, T \,\in\, \mathcal{B}\,(\,H \,\otimes\, K\,)$\, \cite{Folland}.\\

\begin{theorem}\cite{Folland}\label{th1.1}
Suppose \,$Q,\, Q^{\prime} \,\in\, \mathcal{B}\,(\,H\,)$\, and \,$T,\, T^{\prime} \,\in\, \mathcal{B}\,(\,K\,)$, then \begin{itemize}
\item[(I)]\hspace{.2cm} \,$Q \,\otimes\, T \,\in\, \mathcal{B}\,(\,H \,\otimes\, K\,)$\, and \,$\left\|\,Q \,\otimes\, T\,\right\| \,=\, \|\,Q\,\|\; \|\,T\,\|$.
\item[(II)]\hspace{.2cm} \,$\left(\,Q \,\otimes\, T\,\right)\,(\,f \,\otimes\, g\,) \,=\, Q\,(\,f\,) \,\otimes\, T\,(\,g\,)$\, for all \,$f \,\in\, H,\, g \,\in\, K$.
\item[(III)]\hspace{.2cm} $\left(\,Q \,\otimes\, T\,\right)\,\left(\,Q^{\,\prime} \,\otimes\, T^{\,\prime}\,\right) \,=\, (\,Q\,Q^{\,\prime}\,) \,\otimes\, (\,T\,T^{\,\prime}\,)$. 
\item[(IV)]\hspace{.2cm} \,$Q \,\otimes\, T$\, is invertible if and only if \,$Q$\, and \,$T$\, are invertible, in which case \,$\left(\,Q \,\otimes\, T\,\right)^{\,-\, 1} \,=\, \left(\,Q^{\,-\, 1} \,\otimes\, T^{\,-\, 1}\,\right)$.
\item[(V)]\hspace{.2cm} \,$\left(\,Q \,\otimes\, T\,\right)^{\,\ast} \,=\, \left(\,Q^{\,\ast} \,\otimes\, T^{\,\ast}\,\right)$.    
\end{itemize}
\end{theorem}

%=====================================
\section{Construction of $g$-fusion frame in tensor product of Hilbert spaces}
%=====================================

\smallskip\hspace{.6 cm} In this section, we discuss the generalized fusion frame or \,$g$-fusion frame and the canonical dual \,$g$-fusion frame in tensor product of Hilbert spaces. 

\begin{definition}
Let \,$\left\{\,v_{\,i}\,\right\}_{\, i \,\in\, I},\;\left\{\,w_{\,j}\,\right\}_{ j \,\in\, J}$\, be two families of positive weights i\,.\,e., \,$v_{\,i} \,>\, 0\, \;\forall\; i \,\in\, I,\; \,w_{\,j} \,>\, 0\, \;\forall\; j \,\in\, J$\, and \,$\Lambda_{i} \,\otimes\, \Gamma_{j} \,\in\, \mathcal{B}\,(\,H \,\otimes\, K,\, H_{i} \,\otimes\, K_{j}\,)$\, for each \,$i \,\in\, I$\, and \,$j \,\in\, J$.\;Then the family \,$\Lambda \,\otimes\, \Gamma \,=\, \left\{\,\left(\,V_{i} \,\otimes\, W_{j},\, \Lambda_{i} \,\otimes\, \Gamma_{j},\, v_{i}\,w_{j}\,\right)\,\right\}_{\,i,\,j}$\, is said to be a generalized fusion frame or g-fusion frame for \,$H \,\otimes\, K$\, with respect to \,$\left\{\,H_{i} \,\otimes\, K_{j}\,\right\}_{\,i,\,j}$\, if there exist constants \,$0 \,<\, A \,\leq\, B \,<\, \infty$\, such that for all \,$f \,\otimes\, g \,\in\, H \,\otimes\, K$
\begin{equation}\label{eqn2.1}
A\, \left\|\,f \,\otimes\, g\,\right\|^{\,2} \,\leq\, \sum\limits_{i,\, j}\,v^{\,2}_{\,i}\,w^{\,2}_{\,j}\,\left\|\left(\,\Lambda_{i} \,\otimes\, \Gamma_{j}\,\right)\,P_{\,V_{\,i} \,\otimes\, W_{\,j}}\,(\,f \,\otimes\, g\,)\,\right\|^{\,2} \,\leq\, B\, \left\|\,f \,\otimes\, g\,\right\|^{\,2}
\end{equation}
where \,$P_{\,V_{i} \,\otimes\, W_{j}}$\, is the orthogonal projection of \,$H \,\otimes\, K$\, onto \,$V_{\,i} \,\otimes\, W_{\,j}$.\;The constants \,$A$\, and \,$B$\, are called the frame bounds of \,$\Lambda \,\otimes\, \Gamma$.\;If \,$A \,=\, B$\, then it is called a tight g-fusion frame.\;If the family \,$\Lambda \,\otimes\, \Gamma$\, satisfies the inequality
\[\sum\limits_{i,\, j}\,v^{\,2}_{\,i}\,w^{\,2}_{\,j}\,\left\|\left(\,\Lambda_{i} \,\otimes\, \Gamma_{j}\,\right)\,P_{\,V_{i} \,\otimes\, W_{j}}\,(\,f \,\otimes\, g\,)\,\right\|^{\,2} \,\leq\, B\, \left\|\,f \,\otimes\, g\,\right\|^{\,2}\; \;\forall\; f \,\otimes\, g \,\in\, H \,\otimes\, K,\]
then it is called a g-fusion Bessel sequence in \,$H \,\otimes\, K$\, with bound \,$B$.   
\end{definition}

\begin{note}
For \,$i \,\in\, I$\, and \,$j \,\in\, J$, define the space \,$l^{\,2}\,\left(\,\left\{\,H_{i} \,\otimes\, K_{j}\,\right\}\,\right)$
\[\,=\, \left\{\,\left\{\,f_{\,i} \,\otimes\, g_{\,j}\,\right\} \,:\, f_{\,i} \,\otimes\, g_{\,j} \,\in\, H_{i} \,\otimes\, K_{j}, \;\&\; \;\sum\limits_{i,\, j}\,\left\|\,f_{\,i} \,\otimes\, g_{\,j}\,\right\|^{\,2} \,<\, \infty\,\right\}\]
with the inner product \,$\left<\,\left\{\,f_{\,i} \otimes g_{\,j}\,\right\} \,,\, \left\{\,f^{\,\prime}_{\,i} \otimes g^{\,\prime}_{\,j}\,\right\}\,\right>_{l^{\,2}} \,=\, \sum\limits_{i,\, j}\,\left<\,f_{\,i} \otimes g_{\,j}\, \,,\, f^{\,\prime}_{\,i} \otimes g^{\,\prime}_{\,j}\,\right>$
\[=\,\sum\limits_{i,\, j}\,\left<\,f_{\,i} \,,\, f^{\,\prime}_{\,i}\,\right>_{H_{i}}\,\left<\,g_{\,j} \,,\, g^{\,\prime}_{\,j}\,\right>_{K_{j}} \,=\, \left(\,\sum\limits_{\,i \,\in\, I}\,\left<\,f_{\,i} \,,\, f^{\,\prime}_{\,i}\,\right>_{H_{i}}\,\right)\,\left(\,\sum\limits_{\,j \,\in\, J}\,\left<\,g_{\,j} \,,\, g^{\,\prime}_{\,j}\,\right>_{K_{j}}\,\right)\]
\[=\, \left<\,\left\{\,f_{\,i}\,\right\}_{ i \,\in\, I} \,,\, \left\{\,f^{\,\prime}_{\,i}\,\right\}_{ i \,\in\, I}\,\right>_{l^{\,2}\left(\,\left\{\,H_{i}\,\right\}_{i \,\in\, I}\,\right)}\,\left<\,\left\{\,g_{\,j}\,\right\}_{ j \,\in\, J} \,,\, \left\{\,g^{\,\prime}_{\,j}\,\right\}_{ j \,\in\, J}\,\right>_{l^{\,2}\left(\,\left\{\,K_{j}\,\right\}_{ j \,\in\, J}\,\right)}.\]
The space \,$l^{\,2}\,\left(\,\left\{\,H_{i} \,\otimes\, K_{j}\,\right\}\,\right)$\, is complete with the above inner product.\;Then it becomes a Hilbert space with respect to the above inner product.  
\end{note}

\begin{note}
Since \,$\left\{\,V_{i}\,\right\}_{ i \,\in\, I},\; \left\{\,W_{j}\,\right\}_{ j \,\in\, J}$\, and \,$\left\{\,V_{i} \,\otimes\, W_{j} \right\}_{i,\,j}$\, are the families of closed subspaces of the Hilbert spaces \,$H,\, K$\, and \,$H \,\otimes\, K$, respectively, it is easy to verify that \,$P_{\,V_{i} \,\otimes\, W_{j}} \,=\, P_{\,V_{i}} \,\otimes\, P_{\,W_{j}}$.
\end{note}

\begin{theorem}\label{th1.2}
The families \,$\Lambda \,=\, \left\{\,\left(\,V_{i},\, \Lambda_{i},\, v_{\,i}\,\right)\,\right\}_{\, i \,\in\, I}$\, and \,$\Gamma \,=\, \left\{\,\left(\,W_{j},\, \Gamma_{j},\, w_{\,j}\,\right)\,\right\}_{ j \,\in\, J}$\, are g-fusion frames for \,$H$\, and \,$K$\, with respect to \,$\left\{\,H_{i}\,\right\}_{\, i \,\in\, I}$\, and \,$\left\{\,K_{j}\,\right\}_{\, j \,\in\, J}$, respectively if and only if the family \,$\Lambda \,\otimes\, \Gamma \,=\, \left\{\,\left(\,V_{i} \,\otimes\, W_{j},\, \Lambda_{i} \,\otimes\, \Gamma_{j},\, v_{\,i}\,w_{\,j}\,\right)\,\right\}_{\,i,\,j}$\, is a g-fusion frame for \,$H \,\otimes\, K$\, with respect to \,$\left\{\,H_{i} \,\otimes\, K_{j}\,\right\}_{\,i,\,j}$. 
\end{theorem}

\begin{proof}
First we suppose that \,$\Lambda$\, and \,$\Gamma$\, are \,$g$-fusion frames for \,$H$\, and \,$K$.\;Then there exist positive constants \,$(\,A,\, B\,)$\, and \,$(\,C,\, D\,)$\, such that
\begin{equation}\label{eq1}
A \,\left\|\,f \,\right\|_{\,1}^{\,2} \,\leq\, \sum\limits_{\,i \,\in\, I}\, v_{i}^{\,2} \,\left\|\,\Lambda_{i}\,P_{\,V_{i}}\,(\,f\,) \,\right\|_{\,1}^{\,2}  \,\leq\, B\, \left\|\, f \, \right\|_{\,1}^{\,2}\; \;\forall\; f \,\in\, H
\end{equation}
\begin{equation}\label{eq1.1}
C\,\left\|\,g \,\right\|_{\,2}^{\,2} \,\leq\, \sum\limits_{\,j \,\in\, J}\, w_{j}^{\,2}\, \left\|\,\Gamma_{j}\,P_{\,W_{j}}\,(\,g\,) \,\right\|_{\,2}^{\,2} \,\leq\, D\,\left\|\, g \, \right\|_{\,2}^{\,2}\; \;\forall\; g \,\in\, K.
\end{equation}
Multiplying (\ref{eq1}) and (\ref{eq1.1}), and using the definition of norm on \,$H \,\otimes\, K$, we get
\[A\,C\,\left\|\,f \,\right\|_{\,1}^{\,2}\,\left\|\,g \,\right\|_{\,2}^{\,2} \,\leq\, \left(\,\sum\limits_{\,i \,\in\, I}\, v_{i}^{\,2} \,\left\|\,\Lambda_{i}\,P_{\,V_{i}}\,(\,f\,) \,\right\|_{\,1}^{\,2}\,\right)\,\left(\,\sum\limits_{\,j \,\in\, J}\, w_{j}^{\,2}\, \left\|\,\Gamma_{j}\,P_{\,W_{j}}\,(\,g\,) \,\right\|_{\,2}^{\,2}\,\right)\]
\[ \,\leq\, B\,D\,\left\|\,f \,\right\|_{\,1}^{\,2}\,\left\|\,g \,\right\|_{\,2}^{\,2}.\]
\[\Rightarrow\, A\,C\left\|\,f \,\otimes\, g\,\right\|^{\,2} \,\leq\, \sum\limits_{i,\, j}v^{\,2}_{\,i}\,w^{\,2}_{\,j}\,\left\|\,\Lambda_{i}\,P_{\,V_{i}}\,(\,f\,) \,\right\|_{\,1}^{\,2}\,\left\|\,\Gamma_{j}\,P_{\,W_{j}}\,(\,g\,) \,\right\|_{\,2}^{\,2} \,\leq\, B\,D\left\|\,f \,\otimes\, g\,\right\|^{\,2}\]
\[\Rightarrow\, A\,C\left\|\,f \,\otimes\, g\,\right\|^{\,2} \,\leq\, \sum\limits_{i,\, j}v^{\,2}_{\,i}\,w^{\,2}_{\,j}\,\left\|\,\Lambda_{i}\,P_{\,V_{i}}\,(\,f\,) \otimes\, \Gamma_{j}\,P_{\,W_{j}}\,(\,g\,) \,\right\|^{\,2} \,\leq\, B\,D \left\|\,f \,\otimes\, g\,\right\|^{\,2}.\]
Therefore, for all \,$f \,\otimes\, g \,\in\, H \,\otimes\, K$, using the Theorem (\ref{th1.1}), we get
\[A\,C\left\|\,f \otimes g\,\right\|^{\,2} \,\leq\, \sum\limits_{i,\, j}v^{\,2}_{\,i}\,w^{\,2}_{\,j}\left\|\,\left(\,\Lambda_{i} \otimes \Gamma_{j}\,\right)\,\left(\,P_{\,V_{i}} \otimes P_{\,W_{j}}\,\right)\,(\,f \otimes g\,) \,\right\|^{\,2} \,\leq\, B\,D\, \left\|\,f \otimes g\,\right\|^{\,2}\]
\[\Rightarrow\, A\,C\, \left\|\,f \otimes g\,\right\|^{\,2} \,\leq\, \sum\limits_{i,\, j}\,v^{\,2}_{\,i}\,w^{\,2}_{\,j}\,\left\|\,\left(\,\Lambda_{i} \otimes \Gamma_{j}\,\right)\,P_{\,V_{i} \,\otimes\, W_{j}}\,(\,f \otimes g\,)\,\right\|^{\,2} \,\leq\, B\,D\,\left\|\,f \otimes g\,\right\|^{\,2}.\]
Hence, \,$\Lambda \,\otimes\, \Gamma$\, is a \,$g$-fusion frame for \,$H \,\otimes\, K$\, with respect to \,$\left\{\,H_{i} \,\otimes\, K_{j}\,\right\}_{\,i,\,j}$\, with bounds \,$A\,C$\, and \,$B\,D$.\\ 

Conversely, suppose that \,$\Lambda \,\otimes\, \Gamma$\, is a \,$g$-fusion frame for \,$H \,\otimes\, K$\, with respect to \,$\left\{\,H_{i} \,\otimes\, K_{j}\,\right\}_{\,i,\,j}$.\;Then there exist constants \,$A,\, B \,>\, 0$\, such that for all \,$f \,\otimes\, g \,\in\, H \,\otimes\, K \,-\, \{\,\theta \,\otimes\, \theta\,\}$,
\[A\, \left\|\,f \,\otimes\, g\,\right\|^{\,2} \,\leq\, \sum\limits_{i,\, j}\,v^{\,2}_{\,i}\,w^{\,2}_{\,j}\,\left\|\left(\,\Lambda_{i} \,\otimes\, \Gamma_{j}\,\right)\,P_{\,V_{\,i} \,\otimes\, W_{\,j}}\,(\,f \,\otimes\, g\,)\,\right\|^{\,2} \,\leq\, B\, \left\|\,f \,\otimes\, g\,\right\|^{\,2}.\]
\[\Rightarrow\, A\,\left\|\,f \,\otimes\, g\,\right\|^{\,2} \,\leq\, \sum\limits_{i,\, j}\,v^{\,2}_{\,i}\,w^{\,2}_{\,j}\,\left\|\,\Lambda_{i}\,P_{\,V_{i}}\,(\,f\,) \otimes\, \Gamma_{j}\,P_{\,W_{j}}\,(\,g\,) \,\right\|^{\,2} \,\leq\, B\,\left\|\,f \,\otimes\, g\,\right\|^{\,2}.\]
Using (\ref{eq1.001}) and (\ref{eq1.0001}), we obtain 
\[A\left\|\,f \,\right\|_{\,1}^{\,2}\left\|\,g \,\right\|_{\,2}^{\,2} \,\leq\, \left(\,\sum\limits_{\,i \,\in\, I} v_{i}^{\,2}\left\|\,\Lambda_{i}\,P_{\,V_{i}}\,(\,f\,) \,\right\|_{\,1}^{\,2}\,\right)\left(\,\sum\limits_{\,j \,\in\, J} w_{j}^{\,2} \left\|\,\Gamma_{j}\,P_{\,W_{j}}\,(\,g\,) \,\right\|_{\,2}^{\,2}\,\right) \,\leq\, B\left\|\,f \,\right\|_{\,1}^{\,2}\left\|\,g \,\right\|_{\,2}^{\,2}.\]
Since \,$f \,\otimes\, g$\, is non-zero vector, \,$f$\, and \,$g$\, are also non-zero vectors and therefore \,$\sum\limits_{\,i \,\in\, I}\, v_{i}^{\,2} \,\left\|\,\Lambda_{i}\,P_{\,V_{i}}\,(\,f\,) \,\right\|_{\,1}^{\,2}$\, and \,$\sum\limits_{\,j \,\in\, J}\, w_{j}^{\,2}\, \left\|\,\Gamma_{j}\,P_{\,W_{j}}\,(\,g\,) \,\right\|_{\,2}^{\,2}$\, are non-zero.
\[\Rightarrow\, \dfrac{A\,\left\|\,g \,\right\|_{\,2}^{\,2}\,\left\|\,f \,\right\|_{\,1}^{\,2}}{\sum\limits_{\,j \,\in\, J}\, w_{j}^{\,2}\, \left\|\,\Gamma_{j}\,P_{\,W_{j}}\,(\,g\,) \,\right\|_{\,2}^{\,2}} \,\leq\, \sum\limits_{\,i \,\in\, I}\, v_{i}^{\,2} \,\left\|\,\Lambda_{i}\,P_{\,V_{i}}\,(\,f\,) \,\right\|_{\,1}^{\,2} \,\leq\, \dfrac{B\,\left\|\,g \,\right\|_{\,2}^{\,2}\,\left\|\,f \,\right\|_{\,1}^{\,2}}{\sum\limits_{\,j \,\in\, J}\, w_{j}^{\,2}\, \left\|\,\Gamma_{j}\,P_{\,W_{j}}\,(\,g\,) \,\right\|_{\,2}^{\,2}}\]
\[\Rightarrow\, A_{1} \,\left\|\,f \,\right\|_{\,1}^{\,2} \,\leq\, \sum\limits_{\,i \,\in\, I}\, v_{i}^{\,2} \,\left\|\,\Lambda_{i}\,P_{\,V_{i}}\,(\,f\,) \,\right\|_{\,1}^{\,2} \,\leq\, B_{1}\, \left\|\, f \, \right\|_{\,1}^{\,2}\; \;\forall\; f \,\in\, H,\]
where \,$A_{1} \,=\, \dfrac{A\,\left\|\,g \,\right\|_{\,2}^{\,2}}{\sum\limits_{\,j \,\in\, J}\, w_{j}^{\,2}\, \left\|\,\Gamma_{j}\,P_{\,W_{j}}\,(\,g\,) \,\right\|_{\,2}^{\,2}}$\, and \,$B_{1} \,=\, \dfrac{B\,\left\|\,g \,\right\|_{\,2}^{\,2}}{\sum\limits_{\,j \,\in\, J}\, w_{j}^{\,2}\, \left\|\,\Gamma_{j}\,P_{\,W_{j}}\,(\,g\,) \,\right\|_{\,2}^{\,2}}$.\;This shows that \,$\Lambda$\, is a \,$g$-fusion frame for \,$H$\, with respect to \,$\left\{\,H_{i}\,\right\}_{\, i \,\in\, I}$.\;Similarly, it can be shown that \,$\Gamma$\, is \,$g$-fusion frame for \,$K$\, with respect to \,$\left\{\,K_{j}\,\right\}_{\, j \,\in\, J}$.\;This completes the proof. 
\end{proof}

\begin{note}
Let \,$\Lambda \,\otimes\, \Gamma \,=\, \left\{\,\left(\,V_{i} \,\otimes\, W_{j},\, \Lambda_{i} \,\otimes\, \Gamma_{j},\, v_{\,i}\,w_{\,j}\,\right)\,\right\}_{\,i,\,j}$\, be a g-fusion frame for the Hilbert space \,$H \,\otimes\, K$.\;According to the definition (\ref{defn1}), the synthesis operator \,$T_{\,\Lambda \,\otimes\, \Gamma} \,:\, l^{\,2}\,\left(\,\left\{\,H_{i} \,\otimes\, K_{j}\,\right\}\,\right) \,\to\, H \,\otimes\, K$\, is given by 
\[T_{\,\Lambda \,\otimes\, \Gamma}\,\left(\,\left\{\,f_{\,i} \,\otimes\, g_{\,j}\,\right\}\,\right) \,=\, \sum\limits_{i,\, j}\,v_{\,i}\,w_{\,j}\,P_{\,V_{i} \,\otimes\, W_{j}}\,\left(\,\Lambda_{i} \,\otimes\, \Gamma_{j}\,\right)^{\,\ast}\,\left(\,f_{\,i} \,\otimes\, g_{\,j}\,\right)\]
for all \,$\left\{\,f_{\,i} \,\otimes\, g_{\,j}\,\right\} \,\in\, l^{\,2}\,\left(\,\left\{\,H_{i} \,\otimes\, K_{j}\,\right\}\,\right)$, and similarly the frame operator \,$S_{\,\Lambda \,\otimes\, \Gamma} \,:\, H \,\otimes\, K \,\to\, H \,\otimes\, K$\, is described by
\[S_{\,\Lambda \,\otimes\, \Gamma}\,(\,f \,\otimes\, g\,) \,=\, \sum\limits_{i,\, j}\,v^{\,2}_{\,i}\,w^{\,2}_{\,j}\,P_{\,V_{i} \,\otimes\, W_{j}}\,\left(\,\Lambda_{i} \,\otimes\, \Gamma_{j}\,\right)^{\,\ast}\,\left(\,\Lambda_{i} \,\otimes\, \Gamma_{j}\,\right)\,P_{\,V_{i} \,\otimes\, W_{j}}\,(\,f \,\otimes\, g\,)\]
for all \,$f \,\otimes\, g \,\in\, H \,\otimes\, K$.  
\end{note}

\begin{theorem}
If \,$S_{\Lambda}, S_{\Gamma}$\, and \,$S_{\Lambda \,\otimes\, \Gamma}$\, are the associated \,$g$-fusion frame operators and \,$T_{\Lambda}, T_{\Gamma}$\, and \,$T_{\Lambda \,\otimes\, \Gamma}$\, are the synthesis operators of g-fusion frames \,$\Lambda,\, \Gamma$\, and \,$\Lambda \,\otimes\, \Gamma$\, for \,$H,\, K$\, and \,$H \,\otimes\, K$, respectively, then \,$S_{\Lambda \,\otimes\, \Gamma} \,=\, S_{\Lambda} \,\otimes\, S_{\Gamma}\,,\; S^{\,-\, 1}_{\Lambda \,\otimes\, \Gamma} \,=\, S^{\,-\, 1}_{\Lambda} \,\otimes\, S^{\,-\, 1}_{\Gamma}$, and \,$T_{\Lambda \,\otimes\, \Gamma} \,=\, T_{\Lambda} \,\otimes\, T_{\Gamma},\; T^{\,\ast}_{\Lambda \,\otimes\, \Gamma} \,=\, T^{\,\ast}_{\Lambda} \,\otimes\, T^{\,\ast}_{\Gamma}$.
\end{theorem}

\begin{proof}
For each \,$f \,\otimes\, g \,\in\, H \,\otimes\, K$, we have
\[S_{\,\Lambda \,\otimes\, \Gamma}\,(\,f \,\otimes\, g\,) \,=\, \sum\limits_{i,\, j}\,v^{\,2}_{\,i}\,w^{\,2}_{\,j}\,P_{\,V_{i} \,\otimes\, W_{j}}\,\left(\,\Lambda_{i} \,\otimes\, \Gamma_{j}\,\right)^{\,\ast}\,\left(\,\Lambda_{i} \,\otimes\, \Gamma_{j}\,\right)\,P_{\,V_{i} \,\otimes\, W_{j}}\,(\,f \,\otimes\, g\,)\]
\[\,=\, \sum\limits_{i,\, j}\,v^{\,2}_{\,i}\,w^{\,2}_{\,j}\,\left(\,P_{\,V_{i}} \,\otimes\, P_{\,W_{j}}\,\right)\,\left(\,\Lambda^{\,\ast}_{i} \,\otimes\, \Gamma^{\,\ast}_{j}\,\right)\,\left(\,\Lambda_{i} \,\otimes\, \Gamma_{j}\,\right)\,\left(\,P_{\,V_{i}} \,\otimes\, P_{\,W_{j}}\,\right)\,(\,f \,\otimes\, g\,)\hspace{.4cm}\]
\[ \,=\, \sum\limits_{i,\, j}\,v^{\,2}_{\,i}\,w^{\,2}_{\,j}\,\left(\,P_{\,V_{i}}\,\Lambda^{\,\ast}_{i}\,\Lambda_{i}\,P_{\,V_{i}}\,(\,f\,) \,\otimes\, P_{\,W_{j}}\,\Gamma^{\,\ast}_{j}\,\Gamma_{j}\,P_{\,W_{j}}\,(\,g\,)\,\right)\;[\;\text{by Theorem (\ref{th1.1})}\;]\]
\[\,=\, \left(\,\sum\limits_{\,i \,\in\, I}\,v^{\,2}_{\,i}\,P_{\,V_{i}}\,\Lambda^{\,\ast}_{i}\,\Lambda_{i}\,P_{\,V_{i}}\,(\,f\,)\,\right) \,\otimes\, \left(\,\sum\limits_{\,j \,\in\, J}\,w^{\,2}_{j}\,P_{\,W_{j}}\,\Gamma^{\,\ast}_{j}\,\Gamma_{j}\,P_{\,W_{j}}\,(\,g\,)\,\right)\hspace{1.7cm}\]
\[=\, S_{\Lambda}\,(\,f\,) \,\otimes\, S_{\Gamma}\,(\,g\,) \,=\, \left(\,S_{\Lambda} \,\otimes\, S_{\Gamma}\,\right)\,(\,f \,\otimes\, g\,).\hspace{5.5cm}\] 
This implies that \,$S_{\,\Lambda \,\otimes\, \Gamma} \,=\, S_{\Lambda} \,\otimes\, S_{\Gamma}$.\;Since \,$S_{\Lambda}\, \;\text{and}\; \,S_{\Gamma}$\, are invertible, by \,$(\,IV\,)$\, of Theorem (\ref{th1.1}), it follows that \,$S^{\,-\, 1}_{\Lambda \,\otimes\, \Gamma} \,=\, S^{\,-\, 1}_{\Lambda} \,\otimes\, S^{\,-\, 1}_{\Gamma}$.\;On the other hand, for all \,$\left\{\,f_{\,i} \,\otimes\, g_{\,j}\,\right\} \,\in\, l^{\,2}\,\left(\,\left\{\,H_{i} \,\otimes\, K_{j}\,\right\}\,\right)$, we have
\[T_{\,\Lambda \,\otimes\, \Gamma}\,\left(\,\left\{\,f_{\,i} \,\otimes\, g_{\,j}\,\right\}\,\right) \,=\, \sum\limits_{i,\, j}\,v_{\,i}\,w_{\,j}\,P_{\,V_{i} \,\otimes\, W_{j}}\,\left(\,\Lambda_{i} \,\otimes\, \Gamma_{j}\,\right)^{\,\ast}\,\left(\,f_{\,i} \,\otimes\, g_{\,j}\,\right)\]
\[=\, \left(\,\sum\limits_{\,i \,\in\, I}\,v_{\,i}\,P_{\,V_{i}}\,\Lambda^{\,\ast}_{i}\,f_{\,i}\,\right) \,\otimes\, \left(\,\sum\limits_{\,j \,\in\, J}\,w_{j}\,P_{\,W_{j}}\,\Gamma^{\,\ast}_{j}\,g_{\,j}\,\right)\]
\[\hspace{1.5cm} \,=\, T_{\Lambda}\,\left(\,\left\{\,f_{\,i}\,\right\}\,\right) \,\otimes\, T_{\Gamma}\,\left(\,\left\{\,g_{\,j}\,\right\}\,\right) \,=\, \left(\,T_{\Lambda} \,\otimes\, T_{\Gamma}\,\right)\, \left(\,\left\{\,f_{\,i} \,\otimes\, g_{\,j}\,\right\}\,\right).\]
This shows that \,$T_{\,\Lambda \,\otimes\, \Gamma} \,=\, T_{\Lambda} \,\otimes\, T_{\Gamma}$.\;Again by \,$(\,V\,)$\, of Theorem (\ref{th1.1}), it follows that \,$T^{\,\ast}_{\Lambda \,\otimes\, \Gamma} \,=\, T^{\,\ast}_{\Lambda} \,\otimes\, T^{\,\ast}_{\Gamma}$.\;This completes the proof.     
\end{proof}

\begin{theorem}\label{th2}
Let \,$\Lambda \,=\, \left\{\,\left(\,V_{i},\, \Lambda_{i},\, v_{\,i}\,\right)\,\right\}_{\, i \,\in\, I}$\, and \,$\Gamma \,=\, \left\{\,\left(\,W_{j},\, \Gamma_{j},\, w_{\,j}\,\right)\,\right\}_{ j \,\in\, J}$\, be g-fusion frames for \,$H$\, and \,$K$\, with g-fusion frame operators \,$S_{\Lambda}$\, and \,$S_{\Gamma}$, respectively.\;Then \,$\Theta \,=\, \left\{\,S^{\,-\, 1}_{\,\Lambda \,\otimes\, \Gamma}\,\left(\,V_{i} \,\otimes\, W_{j}\,\right),\, \left(\,\Lambda_{i} \,\otimes\, \Gamma_{j}\,\right)\,P_{\,V_{i} \,\otimes\, W_{j}}\,S^{\,-\, 1}_{\,\Lambda \,\otimes\, \Gamma}\,,\, v_{\,i}\,w_{\,j}\,\right\}_{\,i,\,j}$\, is a g-fusion frame for \,$H \,\otimes\, K$.  
\end{theorem}

\begin{proof}
Let \,$(\,A,\, B\,)$\, and \,$(\,C,\, D\,)$\, be the \,$g$-fusion frame bounds of \,$\Lambda$\, and \,$\Gamma$, respectively.\;Now, for each \,$f \,\otimes\, g \,\in\, H \,\otimes\, K$, we have
\[\sum\limits_{i,\, j}\,v^{\,2}_{\,i}\,w^{\,2}_{\,j}\,\left\|\left(\,\Lambda_{i} \,\otimes\, \Gamma_{j}\,\right)\,P_{\,V_{i} \,\otimes\, W_{j}}\,S^{\,-\, 1}_{\Lambda \,\otimes\, \Gamma}\,P_{\,S^{\,-\, 1}_{\Lambda \,\otimes\, \Gamma}\,\left(\,V_{i} \,\otimes\, W_{j}\,\right)}\,(\,f \,\otimes\, g\,)\,\right\|^{\,2}\]
\[=\, \sum\limits_{i,\, j}v^{\,2}_{\,i}\,w^{\,2}_{\,j}\left\|\left(\,\Lambda_{i} \otimes \Gamma_{j}\,\right)\,\left(\,P_{\,V_{\,i}} \otimes P_{W_{j}}\,\right)\,\left(\,S^{\,-\, 1}_{\Lambda} \otimes S^{\,-\, 1}_{\Gamma}\,\right)\,\left(\,P_{\,S^{\,-\, 1}_{\Lambda}\,V_{i}} \otimes P_{\,S^{\,-\, 1}_{\Gamma}\,W_{j}}\,\right)\,(\,f \otimes g\,)\,\right\|^{\,2}\]
\[=\, \sum\limits_{i,\, j}\,v^{\,2}_{\,i}\,w^{\,2}_{\,j}\,\left\|\,\Lambda_{i}\,P_{\,V_{\,i}}\,S^{\,-\, 1}_{\Lambda}\,P_{\,S^{\,-\, 1}_{\Lambda}\,V_{i}}\,(\,f\,) \otimes \Gamma_{j}\,P_{W_{j}}\,S^{\,-\, 1}_{\Gamma}\,P_{\,S^{\,-\, 1}_{\Gamma}\,W_{j}}\,(\,g\,)\,\right\|^{\,2}\;[\;\text{by Theorem (\ref{th1.1})}\;]\]
\[=\, \left(\,\sum\limits_{\,i \,\in\, I}\,v^{\,2}_{\,i}\,\left\|\,\Lambda_{i}\,P_{\,V_{\,i}}\,S^{\,-\, 1}_{\Lambda}\,P_{\,S^{\,-\, 1}_{\Lambda}\,V_{i}}\,(\,f\,)\,\right\|_{1}^{\,2}\,\right)\,\left(\,\sum\limits_{\,j \,\in\, J}\,w^{\,2}_{\,j}\,\left\|\,\Gamma_{j}\,P_{W_{j}}\,S^{\,-\, 1}_{\Gamma}\,P_{\,S^{\,-\, 1}_{\Gamma}\,W_{j}}\,(\,g\,)\,\right\|^{\,2}_{2}\,\right)\]
\[=\, \left(\,\sum\limits_{\,i \,\in\, I}\,v^{\,2}_{\,i}\,\left\|\,\Lambda_{i}\,P_{\,V_{\,i}}\,S^{\,-\, 1}_{\Lambda}\,(\,f\,)\,\right\|_{1}^{\,2}\,\right)\,\left(\,\sum\limits_{\,j \,\in\, J}\,w^{\,2}_{\,j}\,\left\|\,\Gamma_{j}\,P_{W_{j}}\,S^{\,-\, 1}_{\Gamma}\,(\,g\,)\,\right\|^{\,2}_{2}\,\right)\; [\;\text{by Theorem (\ref{th0.001})}\;]\]
\[\leq\, B\,\left\|\,S^{\,-\, 1}_{\Lambda}\,(\,f\,)\,\right\|^{\,2}_{1}\,D\,\left\|\,S^{\,-\, 1}_{\Gamma}\,(\,g\,)\,\right\|^{\,2}_{2}\; \;[\;\text{since $\Lambda,\,\Gamma$ are $g$-fusion frames}\;]\hspace{2.3cm}\]
\[\leq\, B\,D\,\left\|\,S^{\,-\, 1}_{\Lambda}\,\right\|^{\,2}\,\|\,f\,\|^{\,2}_{1}\,\left\|\,S^{\,-\, 1}_{\Gamma}\,\right\|^{\,2}\,\|\,g\,\|_{2} \,\leq\, \dfrac{B\,D}{A^{\,2}\,C^{\,2}}\,\|\,f \,\otimes\, g\,\|^{\,2}.\hspace{3.6cm}\]
\[[\;\text{since $B^{\,-1}\,I_{H} \,\leq\, S^{\,-1}_{\Lambda} \,\leq\, A^{\,-1}\,I_{H}$ and $D^{\,-1}\,I_{K} \,\leq\, S^{\,-1}_{\Gamma} \,\leq\, C^{\,-1}\,I_{K}$}\;].\]
On the other hand, again using Theorem (\ref{th0.001})
\[\|\,f \,\otimes\, g\,\|^{\,4} \,=\, \left|\,\left<\,f \,\otimes\, g \,,\, f \,\otimes\, g\,\right>\,\right|^{\,2} \,=\, \left|\,\left<\,f \,,\, f\,\right>_{1}\,\right|^{\,2}\, \left|\,\left<\,g \,,\, g\,\right>_{2}\,\right|^{\,2}\]
\[=\, \left|\,\left<\,\sum\limits_{\,i \,\in\, I}\,v^{\,2}_{\,i}\,P_{\,V_{i}}\,\Lambda^{\,\ast}_{i}\,\Lambda_{i}\,P_{\,V_{i}}\,S^{\,-\, 1}_{\Lambda}\,(\,f\,) \,,\, f\,\right>_{1}\,\right|^{\,2}\,\left|\,\left<\,\sum\limits_{\,j \,\in\, J}\,w^{\,2}_{j}\,P_{\,W_{j}}\,\Gamma^{\,\ast}_{j}\,\Gamma_{j}\,P_{\,W_{j}}\,S_{\Gamma}^{\,-\, 1}\,(\,g\,) \,,\, g\,\right>_{2}\,\right|^{\,2}\]
\[=\, \left|\,\sum\limits_{\,i \,\in\, I}\,v^{\,2}_{\,i}\,\left<\,\Lambda_{i}\,P_{\,V_{i}}\,S^{\,-\, 1}_{\Lambda}\,(\,f\,) \,,\, \Lambda_{i}\,P_{\,V_{i}}\,(\,f\,)\,\right>_{1}\,\right|^{\,2}\,\left|\,\sum\limits_{\,j \,\in\, J}\,w^{\,2}_{j}\,\left<\,\Gamma_{j}\,P_{\,W_{j}}\,S_{\Gamma}^{\,-\, 1}\,(\,g\,) \,,\, \Gamma_{j}\,P_{\,W_{j}}\,(\,g\,)\,\right>_{2}\,\right|^{\,2}\hspace{1.2cm}\]
\[\leq\, \left(\,\sum\limits_{\,i \,\in\, I}\,v^{\,2}_{\,i}\,\left\|\,\Lambda_{i}\,P_{\,V_{i}}\,S^{\,-\, 1}_{\Lambda}\,(\,f\,)\,\right\|^{\,2}_{1}\,\right)\,\left(\,\sum\limits_{\,i \,\in\, I}\,v^{\,2}_{\,i}\,\left\|\,\Lambda_{i}\,P_{\,V_{i}}\,(\,f\,)\,\right\|^{\,2}_{1}\,\right)\,\left(\,\sum\limits_{\,j \,\in\, J}\,w^{\,2}_{\,j}\,\left\|\,\Gamma_{j}\,P_{\,W_{j}}\,S_{\Gamma}^{\,-\, 1}\,(\,g\,)\,\right\|^{\,2}_{2}\,\right)\times\]
\[\hspace{4.2cm}\left(\,\sum\limits_{\,j \,\in\, J}\,w^{\,2}_{\,j}\,\left\|\,\Gamma_{j}\,P_{\,W_{j}}\,(\,g\,)\,\right\|^{\,2}_{2}\,\right)\; \;[\;\text{by C-S inequality}\;]\]
\[\leq\, B\,D\,\|\,f\,\|^{\,2}_{1}\,\|\,g\,\|_{2}\,\left(\,\sum\limits_{\,i \,\in\, I}\,v^{\,2}_{\,i}\,\left\|\,\Lambda_{i}\,P_{\,V_{i}}\,S^{\,-\, 1}_{\Lambda}\,P_{\,S^{\,-\, 1}_{\Lambda}\,V_{i}}\,(\,f\,)\,\right\|^{\,2}_{1}\,\right)\,\left(\,\sum\limits_{\,j \,\in\, J}\,w^{\,2}_{\,j}\,\left\|\,\Gamma_{j}\,P_{\,W_{j}}\,S_{\Gamma}^{\,-\, 1}\,P_{\,S^{\,-\, 1}_{\Gamma}\,W_{j}}\,(\,g\,)\,\right\|^{\,2}_{2}\,\right)\]
\[=\, B\,D\,\|\,f \,\otimes\, g\,\|^{\,2}\,\sum\limits_{i,\, j}\,v^{\,2}_{\,i}\,w^{\,2}_{\,j}\,\left\|\,\Lambda_{i}\,P_{\,V_{\,i}}\,S^{\,-\, 1}_{\Lambda}\,P_{\,S^{\,-\, 1}_{\Lambda}\,V_{i}}\,(\,f\,) \,\otimes\, \Gamma_{j}\,P_{W_{\,j}}\,S^{\,-\, 1}_{\Gamma}\,P_{\,S^{\,-\, 1}_{\Gamma}\,W_{j}}\,(\,g\,)\,\right\|^{\,2}\]
\[=\, B\,D\,\|\,f \,\otimes\, g\,\|^{\,2}\,\sum\limits_{i,\, j}\,v^{\,2}_{\,i}\,w^{\,2}_{\,j}\,\left\|\left(\,\Lambda_{i} \,\otimes\, \Gamma_{j}\,\right)\,P_{\,V_{\,i} \,\otimes\, W_{\,j}}\,S^{\,-\, 1}_{\Lambda \,\otimes\, \Gamma}\,P_{\,S^{\,-\, 1}_{\Lambda \,\otimes\, \Gamma}\,\left(\,V_{\,i} \,\otimes\, W_{\,j}\,\right)}\,(\,f \,\otimes\, g\,)\,\right\|^{\,2}.\]This implies that for all \,$f \,\otimes\, g \,\in\, H \,\otimes\, K$, we get 
\[\dfrac{1}{B\,D}\,\|\,f \,\otimes\, g\,\|^{\,2} \,\leq\, \sum\limits_{i,\, j}\,v^{\,2}_{\,i}\,w^{\,2}_{\,j}\,\left\|\left(\,\Lambda_{i} \,\otimes\, \Gamma_{j}\,\right)\,P_{\,V_{\,i} \,\otimes\, W_{\,j}}\,S^{\,-\, 1}_{\Lambda \,\otimes\, \Gamma}\,P_{\,S^{\,-\, 1}_{\Lambda \,\otimes\, \Gamma}\,\left(\,V_{\,i} \,\otimes\, W_{\,j}\,\right)}\,(\,f \,\otimes\, g\,)\,\right\|^{\,2}.\]
Hence, \,$\Theta$\, is a \,$g$-fusion frame for \,$H \,\otimes\, K$\, with bounds \,$\dfrac{1}{B\,D}$\, and \,$\dfrac{B\,D}{A^{\,2}\,C^{\,2}}$. 
\end{proof}

\begin{result}
For the g-fusion frame \,$\Theta$, frame operator is \,$S^{\,-\, 1}_{\,\Lambda \,\otimes\, \Gamma}$.
\end{result}

\begin{proof}
Take \,$\Delta \,=\, \left(\,\Lambda_{i} \,\otimes\, \Gamma_{j}\,\right)\,P_{\,V_{i} \,\otimes\, W_{j}}\,S^{\,-\, 1}_{\,\Lambda \,\otimes\, \Gamma}$.\;Then 
\[\Delta^{\,\ast}\,\Delta \,=\, \left(\,\left(\,\Lambda_{i} \,\otimes\, \Gamma_{j}\,\right)\,P_{\,V_{i} \,\otimes\, W_{j}}\,S^{\,-\, 1}_{\,\Lambda \,\otimes\, \Gamma}\,\right)^{\,\ast}\,\left(\,\Lambda_{i} \,\otimes\, \Gamma_{j}\,\right)\,P_{\,V_{i} \,\otimes\, W_{j}}\,S^{\,-\, 1}_{\,\Lambda \,\otimes\, \Gamma}\]
\[=\, \left(\,S^{\,-\, 1}_{\Lambda} \,\otimes\, S^{\,-\, 1}_{\Gamma}\,\right)\,\left(\,P_{\,V_{i}} \,\otimes\, P_{W_{j}}\,\right)\,\left(\,\Lambda^{\,\ast}_{i} \,\otimes\, \Gamma^{\,\ast}_{j}\,\right)\,\left(\,\Lambda_{i} \,\otimes\, \Gamma_{j}\,\right)\,\left(\,P_{\,V_{i}} \,\otimes\, P_{W_{j}}\,\right)\,\left(\,S^{\,-\, 1}_{\Lambda} \,\otimes\, S^{\,-\, 1}_{\Gamma}\,\right)\]
\[=\, S^{\,-\, 1}_{\Lambda}\,P_{\,V_{i}}\,\Lambda^{\,\ast}_{i}\,\Lambda_{i}\,P_{\,V_{i}}\,S^{\,-\, 1}_{\Lambda} \,\otimes\, S^{\,-\, 1}_{\Gamma}\,P_{W_{j}}\,\Gamma^{\,\ast}_{j}\,\Gamma_{j}\,P_{W_{j}}\,S^{\,-\, 1}_{\Gamma}\; \;[\;\text{by Theorem (\ref{th1.1})}\;].\hspace{2cm}\]
Now, \,$P_{\,S^{\,-\, 1}_{\Lambda \,\otimes\, \Gamma}\,\left(\,V_{\,i} \,\otimes\, W_{j}\,\right)}\,\Delta^{\,\ast}\,\Delta\,P_{\,S^{\,-\, 1}_{\Lambda \,\otimes\, \Gamma}\,\left(\,V_{i} \,\otimes\, W_{j}\,\right)}$
\[=\, \left(\,P_{\,S^{\,-\, 1}_{\Lambda}\,V_{i}} \,\otimes\, P_{\,S^{\,-\, 1}_{\Gamma}\,W_{j}}\,\right)\,\Delta^{\,\ast}\,\Delta\,\left(\,P_{\,S^{\,-\, 1}_{\Lambda}\,V_{i}} \,\otimes\, P_{\,S^{\,-\, 1}_{\Gamma}\,W_{j}}\,\right)\hspace{4.7cm}\]
\[=\, P_{\,S^{\,-\, 1}_{\Lambda}\,V_{i}}\,S^{\,-\, 1}_{\Lambda}\,P_{\,V_{\,i}}\,\Lambda^{\,\ast}_{i}\,\Lambda_{i}\,P_{\,V_{\,i}}\,S^{\,-\, 1}_{\Lambda}\,P_{\,S^{\,-\, 1}_{\Lambda}\,V_{i}} \,\otimes\, P_{\,S^{\,-\, 1}_{\Gamma}\,W_{j}}\,S^{\,-\, 1}_{\Gamma}\,P_{W_{j}}\,\Gamma^{\,\ast}_{j}\,\Gamma_{j}\,P_{W_{j}}\,S^{\,-\, 1}_{\Gamma}\,P_{\,S^{\,-\, 1}_{\Gamma}\,W_{j}}\]
\[=\, \left(\,P_{\,V_{i}}\,S^{\,-\, 1}_{\Lambda}\,\right)^{\,\ast}\,\Lambda^{\,\ast}_{i}\,\Lambda_{i}\,P_{\,V_{i}}\,S^{\,-\, 1}_{\Lambda} \;\otimes\; \left(\,P_{\,W_{j}}\,S^{\,-\, 1}_{\Gamma}\,\right)^{\,\ast}\,\Gamma^{\,\ast}_{j}\,\Gamma_{j}\,P_{\,W_{j}}\,S^{\,-\, 1}_{\Gamma}\; [\;\text{by Theorem (\ref{th0.001})}\;]\]
\[=\, S^{\,-\, 1}_{\Lambda}\,P_{\,V_{\,i}}\,\Lambda^{\,\ast}_{i}\,\Lambda_{i}\,P_{\,V_{\,i}}\,S^{\,-\, 1}_{\Lambda} \;\otimes\; S^{\,-\, 1}_{\Gamma}\,P_{\,W_{j}}\,\Gamma^{\,\ast}_{j}\,\Gamma_{j}\,P_{\,W_{j}}\,S^{\,-\, 1}_{\Gamma}.\hspace{5.2cm}\]
Therefore, for each \,$f \,\otimes\, g \,\in\, H \,\otimes\, K$, we have
\[\sum\limits_{i,\, j}\,v^{\,2}_{\,i}\,w^{\,2}_{\,j}\,P_{\,S^{\,-\, 1}_{\Lambda \,\otimes\, \Gamma}\,\left(\,V_{i} \,\otimes\, W_{j}\,\right)}\,\Delta^{\,\ast}\,\Delta\,P_{\,S^{\,-\, 1}_{\Lambda \,\otimes\, \Gamma}\,\left(\,V_{i} \,\otimes\, W_{j}\,\right)}\,(\,f \,\otimes\, g\,)\]
\[=\, \sum\limits_{i,\, j}\,v^{\,2}_{\,i}\,w^{\,2}_{\,j}\,\left(\,S^{\,-\, 1}_{\Lambda}\,P_{\,V_{i}}\,\Lambda^{\,\ast}_{i}\,\Lambda_{i}\,P_{\,V_{\,i}}\,S^{\,-\, 1}_{\Lambda} \;\otimes\; S^{\,-\, 1}_{\Gamma}\,P_{\,W_{j}}\,\Gamma^{\,\ast}_{j}\,\Gamma_{j}\,P_{\,W_{j}}\,S^{\,-\, 1}_{\Gamma}\,\right)\,(\,f \,\otimes\, g\,)\hspace{1.7cm}\]
\[=\, \left(\,\sum\limits_{\,i \,\in\, I}\,v^{\,2}_{\,i}\,S^{\,-\, 1}_{\Lambda}\,P_{\,V_{\,i}}\,\Lambda^{\,\ast}_{i}\,\Lambda_{i}\,P_{\,V_{\,i}}\,S^{\,-\, 1}_{\Lambda}\,(\,f\,)\,\right) \,\otimes\, \left(\,\sum\limits_{\,j \,\in\, J}\,w^{\,2}_{\,j}\,S^{\,-\, 1}_{\Gamma}\,P_{\,W_{j}}\,\Gamma^{\,\ast}_{j}\,\Gamma_{j}\,P_{\,W_{j}}\,S^{\,-\, 1}_{\Gamma}\,(\,g\,)\,\right)\]
\[=\, S^{\,-\, 1}_{\Lambda}\,S_{\Lambda}\,\left(\,S^{\,-\, 1}_{\Lambda}\,(\,f\,)\,\right) \,\otimes\, S^{\,-\, 1}_{\Gamma}\,S_{\Gamma}\,\left(\,S^{\,-\, 1}_{\Gamma}\,(\,g\,)\,\right)\hspace{6.3cm}\]
\[=\, S^{\,-\, 1}_{\Lambda}\,(\,f\,) \,\otimes\, S^{\,-\, 1}_{\Gamma}\,(\,g\,) \,=\, \left(\,S^{\,-\, 1}_{\Lambda} \,\otimes\, S^{\,-\, 1}_{\Gamma}\,\right)\,(\,f \,\otimes\, g\,) \,=\, S^{\,-\, 1}_{\,\Lambda \,\otimes\, \Gamma}\,(\,f \,\otimes\, g\,).\hspace{2cm}\]
This shows that the corresponding \,$g$-fusion frame operator for \,$\Theta$\, is \,$S^{\,-\, 1}_{\,\Lambda \,\otimes\, \Gamma}$. 
\end{proof}

\begin{remark}
According to the definition (\ref{defn2}), the g-fusion frame \,$\Theta \,=\, $\\
$\left\{\,S^{\,-\, 1}_{\,\Lambda \,\otimes\, \Gamma}\,\left(\,V_{i} \,\otimes\, W_{j}\,\right),\, \left(\,\Lambda_{i} \,\otimes\, \Gamma_{j}\,\right)\,P_{\,V_{i} \,\otimes\, W_{j}}\,S^{\,-\, 1}_{\,\Lambda \,\otimes\, \Gamma},\, v_{\,i}\,w_{\,j}\,\right\}_{\,i,\,j}$\, is called the canonical dual g-fusion frame of \,$\Lambda \,\otimes\, \Gamma$.  
\end{remark}

\begin{theorem}
Let \,$\Lambda \,=\, \left\{\,\left(\,V_{i},\, \Lambda_{i},\, v_{\,i}\,\right)\,\right\}_{\, i \,\in\, I},\, \Lambda^{\,\prime} \,=\, \left\{\,\left(\,V^{\,\prime}_{i},\, \Lambda^{\,\prime}_{i},\, v^{\,\prime}_{\,i}\,\right)\,\right\}_{\, i \,\in\, I}$\, be g-fusion Bessel sequences with bounds \,$B,\, D$, respectively in \,$H$\, and \,$\Gamma \,=\, \left\{\,\left(\,W_{j},\, \Gamma_{j},\, w_{\,j}\,\right)\,\right\}_{ j \,\in\, J}$,\, $\Gamma^{\,\prime} \,=\, \left\{\,\left(\,W^{\,\prime}_{j}, \,\Gamma^{\,\prime}_{j},\, w^{\,\prime}_{\,j}\,\right)\,\right\}_{ j \,\in\, J}$\, be g-fusion Bessel sequences with bounds \,$E,\, F$, respectively in \,$K$.\;Suppose \,$\left(\,T_{\Lambda},\, T_{\Lambda^{\,\prime}}\,\right)$\, and \,$\left(\,T_{\Gamma},\, T_{\Gamma^{\,\prime}}\,\right)$\, are their synthesis operators such that \,$T_{\Lambda^{\,\prime}}\,T^{\,\ast}_{\Lambda} \,=\, I_{H}$\, and \,$T_{\Gamma^{\,\prime}}\,T^{\,\ast}_{\Gamma} \,=\, I_{K}$.\;Then \,$\Lambda \,\otimes\, \Gamma \,=\, \left\{\,\left(\,V_{i} \,\otimes\, W_{j},\, \Lambda_{i} \,\otimes\, \Gamma_{j},\, v_{\,i}\,w_{\,j}\,\right)\,\right\}_{\,i,\,j}$\, and \,$\Lambda^{\,\prime} \,\otimes\, \Gamma^{\,\prime} \,=\, \left\{\,\left(\,V^{\,\prime}_{i} \,\otimes\, W^{\,\prime}_{j},\, \Lambda^{\,\prime}_{i} \,\otimes\, \Gamma^{\,\prime}_{j},\, v^{\,\prime}_{\,i}\,w^{\,\prime}_{\,j}\,\right)\,\right\}_{\,i,\,j}$\, are g-fusion frames for \,$H \,\otimes\, K$.   
\end{theorem}

\begin{proof}
Since \,$(\,\Lambda,\, \Lambda^{\,\prime}\,)$\, and \,$(\,\Gamma,\, \Gamma^{\,\prime}\,)$\, are \,$g$-fusion Bessel sequences in \,$H$\, and \,$K$, respectively, by Theorem (\ref{th1.2}), \,$\Lambda \,\otimes\, \Gamma$\, and \,$\Lambda^{\,\prime} \,\otimes\, \Gamma^{\,\prime}$\, are \,$g$-fusion Bessel sequences with bounds \,$B\,E$\, and \,$D\,F$, respectively in \,$H \,\otimes\, K$.\;Now, for each \,$f \,\otimes\, g \,\in\, H \,\otimes\, K$, 
\[\|\,f \,\otimes\, g\,\|^{\,4} \,=\, \left|\,\left<\,f \,\otimes\, g \,,\, f \,\otimes\, g\,\right>\,\right|^{\,2} \,=\, \left|\,\left<\,f \,,\, f\,\right>_{1}\,\right|^{\,2}\, \left|\,\left<\,g \,,\, g\,\right>_{2}\,\right|^{\,2}\]
\[ \,=\, \left|\,\left<\,T^{\,\ast}_{\Lambda}\,f \,,\, T^{\,\ast}_{\Lambda^{\,\prime}}\,f\,\right>_{1}\,\right|^{\,2}\, \left|\,\left<\,T^{\,\ast}_{\Gamma}\,g \,,\, T^{\,\ast}_{\Gamma^{\,\prime}}\,g\,\right>_{2}\,\right|^{\,2}\leq\, \left\|\,T^{\,\ast}_{\Lambda}\,f\,\right\|^{\,2}_{1}\,\left\|\,T^{\,\ast}_{\Lambda^{\,\prime}}\,f\,\right\|^{\,2}_{1}\,\left\|\,T^{\,\ast}_{\Gamma}\,g\,\right\|^{\,2}_{2}\,\left\|\,T^{\,\ast}_{\Gamma^{\,\prime}}\,g\,\right\|^{\,2}_{2}\]
\[=\, \left(\,\sum\limits_{\,i \,\in\, I}\, (\,v^{\,\prime}_{i}\,)^{\,2}\,\left\|\,\Lambda^{\,\prime}_{i}\,P_{\,V^{\,\prime}_{i}}\,(\,f\,) \,\right\|_{\,1}^{\,2}\,\right)\,\left(\,\sum\limits_{\,j \,\in\, J}\, (\,w^{\,\prime}_{j}\,)^{\,2}\, \left\|\,\Gamma^{\,\prime}_{j}\,P_{\,W^{\,\prime}_{j}}\,(\,g\,) \,\right\|_{\,2}^{\,2}\,\right)\,\times\hspace{2.2cm}\]
\[ \left(\,\sum\limits_{\,i \,\in\, I}\, v_{i}^{\,2} \,\left\|\,\Lambda_{i}\,P_{\,V_{i}}\,(\,f\,) \,\right\|_{\,1}^{\,2}\,\right)\,\left(\,\sum\limits_{\,j \,\in\, J}\, w_{j}^{\,2}\, \left\|\,\Gamma_{j}\,P_{\,W_{j}}\,(\,g\,)\,\right\|_{\,2}^{\,2}\,\right)\]
\[\leq\, D\,F\,\left\|\,f \,\right\|_{\,1}^{\,2}\,\left\|\,g \,\right\|_{\,2}^{\,2}\,\left(\,\sum\limits_{\,i \,\in\, I}\, v_{i}^{\,2} \,\left\|\,\Lambda_{i}\,P_{\,V_{i}}\,(\,f\,)\,\right\|_{\,1}^{\,2}\,\right)\,\left(\,\sum\limits_{\,j \,\in\, J}\, w_{j}^{\,2}\, \left\|\,\Gamma_{j}\,P_{\,W_{j}}\,(\,g\,) \,\right\|_{\,2}^{\,2}\,\right)\hspace{1cm}\]
\[\;[\;\text{since $\Lambda^{\,\prime},\,\Gamma^{\,\prime}$ are $g$-fusion Bessel sequences}\;]\] 
\[=\, D\,F\,\|\,f \,\otimes\, g\,\|^{\,2}\,\sum\limits_{i,\, j}\,v^{\,2}_{\,i}\,w^{\,2}_{\,j}\,\left\|\,\Lambda_{i}\,P_{\,V_{i}}\,(\,f\,) \,\right\|_{\,1}^{\,2}\,\left\|\,\Gamma_{j}\,P_{\,W_{j}}\,(\,g\,) \,\right\|_{\,2}^{\,2}\hspace{3.75cm}\]
\[=\, D\,F\,\|\,f \,\otimes\, g\,\|^{\,2}\,\sum\limits_{i,\, j}\,v^{\,2}_{\,i}\,w^{\,2}_{\,j}\,\left\|\,\Lambda_{i}\,P_{\,V_{i}}\,(\,f\,) \otimes\, \Gamma_{j}\,P_{\,W_{j}}\,(\,g\,) \,\right\|^{\,2}\hspace{4.2cm}\]
\[=\, D\,F\,\|\,f \,\otimes\, g\,\|^{\,2}\,\sum\limits_{i,\, j}\,v^{\,2}_{\,i}\,w^{\,2}_{\,j}\,\left\|\,\left(\,\Lambda_{i} \,\otimes\, \Gamma_{j}\,\right)\,P_{\,V_{i} \,\otimes\, W_{j}}\,(\,f \,\otimes\, g\,)\,\right\|^{\,2}\hspace{3.8cm}\]
\[\Rightarrow\, \dfrac{1}{D\,F}\,\|\,f \,\otimes\, g\,\|^{\,2} \,\leq\, \sum\limits_{i,\, j}\,v^{\,2}_{\,i}\,w^{\,2}_{\,j}\,\left\|\,\left(\,\Lambda_{i} \,\otimes\, \Gamma_{j}\,\right)\,P_{\,V_{i} \,\otimes\, W_{j}}\,(\,f \,\otimes\, g\,)\,\right\|^{\,2}.\hspace{2.7cm}\]
Thus, \,$\Lambda \,\otimes\, \Gamma$\, is a \,$g$-fusion frame for \,$H \,\otimes\, K$.\;Similarly, it can be shown that \,$\Lambda^{\,\prime} \,\otimes\, \Gamma^{\,\prime}$\, is also a \,$g$-fusion frame for \,$H \,\otimes\, K$.     
\end{proof}

%=====================================
\section{Frame operator for a pair of $g$-fusion Bessel sequences in tensor product of Hilbert spaces}
%=====================================

\smallskip\hspace{.6 cm}In this section, the frame operator for a pair of \,$g$-fusion Bessel sequences in \,$H \,\otimes\, K$\, is presented.

\begin{definition}
Let \,$\Lambda \,\otimes\, \Gamma \,=\, \left\{\,\left(\,V_{i} \,\otimes\, W_{j},\, \Lambda_{i} \,\otimes\, \Gamma_{j},\, v_{\,i}\,w_{\,j}\,\right)\,\right\}_{\,i,\,j}$\, and \,$\Lambda^{\,\prime} \,\otimes\, \Gamma^{\,\prime} \,=\, \left\{\,\left(\,V^{\,\prime}_{i} \,\otimes\, W^{\,\prime}_{j},\, \Lambda^{\,\prime}_{i} \,\otimes\, \Gamma^{\,\prime}_{j},\, v^{\,\prime}_{\,i}\,w^{\,\prime}_{\,j}\,\right)\,\right\}_{\,i,\,j}$\, be two g-fusion Bessel sequences in \,$H \,\otimes\, K$.\;Then the operator \,$S \,:\, H \,\otimes\, K \,\to\, H \,\otimes\, K$\, defined by for all \,$f \,\otimes\, g \,\in\, H \,\otimes\, K$,
\[S\,(\,f \,\otimes\, g\,) \,=\, \sum\limits_{i,\, j}\,v_{\,i}\,w_{\,j}\,v^{\,\prime}_{\,i}\,w^{\,\prime}_{\,j}\;P_{\,V_{i} \,\otimes\, W_{j}}\,\left(\,\Lambda_{i} \,\otimes\, \Gamma_{j}\,\right)^{\,\ast}\,\left(\,\Lambda^{\,\prime}_{i} \,\otimes\, \Gamma^{\,\prime}_{j}\,\right)\,P_{\,V^{\,\prime}_{i} \,\otimes\, W^{\,\prime}_{j}}\,(\,f \,\otimes\, g\,)\]is called the frame operator for the pair of g-fusion Bessel sequences \,$\Lambda \,\otimes\, \Gamma$\, and \,$\Lambda^{\,\prime} \,\otimes\, \Gamma^{\,\prime}$. 
\end{definition}

\begin{theorem}
Let \,$S_{\Lambda\,\Lambda^{\,\prime}}$\, and \,$S_{\Gamma\,\Gamma^{\,\prime}}$\, be the frame operators for the pair of g-fusion Bessel sequences \,$\left(\,\Lambda \,=\, \left\{\,\left(\,V_{i},\, \Lambda_{i},\, v_{\,i}\,\right)\,\right\}_{\, i \,\in\, I},\, \Lambda^{\,\prime} \,=\, \left\{\,\left(\,V^{\,\prime}_{i},\, \Lambda^{\,\prime}_{i},\, v^{\,\prime}_{\,i}\,\right)\,\right\}_{\, i \,\in\, I}\,\right)$\, and \,$\left(\,\Gamma \,=\, \left\{\,\left(\,W_{j},\, \Gamma_{j},\, w_{\,j}\,\right)\,\right\}_{ j \,\in\, J},\, \Gamma^{\,\prime} \,=\, \left\{\,\left(\,W^{\,\prime}_{j},\, \Gamma^{\,\prime}_{j},\, w^{\,\prime}_{\,j}\,\right)\,\right\}_{ j \,\in\, J}\,\right)$\, in \,$H$\, and \,$K$, respectively.\;Then \,$S \,=\, S_{\Lambda\,\Lambda^{\,\prime}} \,\otimes\, S_{\Gamma\,\Gamma^{\,\prime}}$.  
\end{theorem}

\begin{proof}
Since \,$S$\, is the associated frame operator for the pair of g-fusion Bessel sequences \,$\Lambda \,\otimes\, \Gamma$\, and \,$\Lambda^{\,\prime} \,\otimes\, \Gamma^{\,\prime}$, for all \,$f \,\otimes\, g \,\in\, H \,\otimes\, K$, we have
\[S\,(\,f \,\otimes\, g\,) \,=\, \sum\limits_{i,\, j}\,v_{\,i}\,w_{\,j}\,v^{\,\prime}_{\,i}\,w^{\,\prime}_{\,j}\;P_{\,V_{i} \,\otimes\, W_{j}}\,\left(\,\Lambda_{i} \,\otimes\, \Gamma_{j}\,\right)^{\,\ast}\,\left(\,\Lambda^{\,\prime}_{i} \,\otimes\, \Gamma^{\,\prime}_{j}\,\right)\,P_{\,V^{\,\prime}_{i} \,\otimes\, W^{\,\prime}_{j}}\,(\,f \,\otimes\, g\,)\]
\[=\, \sum\limits_{i,\, j}\,v_{\,i}\,w_{\,j}\,v^{\,\prime}_{\,i}\,w^{\,\prime}_{\,j}\,\left(\,P_{\,V_{i}} \,\otimes\, P_{\,W_{j}}\,\right)\,\left(\,\Lambda_{i}^{\,\ast} \,\otimes\, \Gamma_{j}^{\,\ast}\,\right)\,\left(\,\Lambda^{\,\prime}_{i} \,\otimes\, \Gamma^{\,\prime}_{j}\,\right)\,\left(\,P_{\,V^{\,\prime}_{i}} \,\otimes\, P_{\,W^{\,\prime}_{j}}\,\right)\,(\,f \,\otimes\, g\,)\]
\[=\, \left(\,\sum\limits_{\,i \,\in\, I}\,v_{\,i}\,v^{\,\prime}_{\,i}\,P_{\,V_{i}}\,\Lambda_{i}^{\,\ast}\,\Lambda_{i}^{\,\prime} \,P_{\,V^{\,\prime}_{i}}\,(\,f\,)\,\right) \,\otimes\, \left(\,\sum\limits_{\,j \,\in\, J}\,w_{\,j}\,w^{\,\prime}_{\,j}\,P_{\,W_{j}}\,\Gamma_{j}^{\,\ast}\,\Gamma_{j}^{\,\prime}\,P_{\,W^{\,\prime}_{j}}\,(\,g\,)\,\right)\hspace{2.7cm}\]
\[=\, S_{\Lambda\,\Lambda^{\,\prime}}\,(\,f\,) \,\otimes\, S_{\Gamma\,\Gamma^{\,\prime}}\,(\,g\,) \,=\, \left(\,S_{\Lambda\,\Lambda^{\,\prime}} \,\otimes\, S_{\Gamma\,\Gamma^{\,\prime}}\,\right)\,(\,f \,\otimes\, g\,).\hspace{5cm}\]
This shows that \,$S \,=\, S_{\Lambda\,\Lambda^{\,\prime}} \,\otimes\, S_{\Gamma\,\Gamma^{\,\prime}}$.   
\end{proof}

\begin{theorem}
The frame operator for the pair of g-fusion Bessel sequences in \,$H \,\otimes\, K$\, is bounded.  
\end{theorem}

\begin{proof}
Let \,$S$\, be the corresponding frame operator for the pair of g-fusion Bessel sequences \,$\Lambda \,\otimes\, \Gamma$\, and \,$\Lambda^{\,\prime} \,\otimes\, \Gamma^{\,\prime}$\, with bounds \,$B_{1}$\, and \,$B_{2}$\, in \,$H \,\otimes\, K$.\;Now, for \,$f \,\otimes\, g,\, f_{\,1} \,\otimes\, g_{\,1} \,\in\, H \,\otimes\, K$, we have \[\left<\,S\,(\,f \,\otimes\, g\,) \;,\; f_{\,1} \,\otimes\, g_{\,1}\,\right> \]
\[\,=\, \left<\,\sum\limits_{i,\, j}\,v_{\,i}\,w_{\,j}\,v^{\,\prime}_{\,i}\,w^{\,\prime}_{\,j}\;P_{\,V_{i} \,\otimes\, W_{j}}\,\left(\,\Lambda_{i} \,\otimes\, \Gamma_{j}\,\right)^{\,\ast}\,\left(\,\Lambda^{\,\prime}_{i} \,\otimes\, \Gamma^{\,\prime}_{j}\,\right)\,P_{\,V^{\,\prime}_{i} \,\otimes\, W^{\,\prime}_{j}}\,(\,f \,\otimes\, g\,) \;,\; f_{\,1} \,\otimes\, g_{\,1}\,\right>\]
\[=\, \sum\limits_{i,\, j}\,v_{\,i}\,w_{\,j}\,v^{\,\prime}_{\,i}\,w^{\,\prime}_{\,j}\,\left<\,P_{\,V_{i}}\,\Lambda_{i}^{\,\ast}\,\Lambda_{i}^{\,\prime}\,P_{\,V^{\,\prime}_{i}}\,(\,f\,) \,\otimes\, P_{\,W_{j}}\,\Gamma_{j}^{\,\ast}\,\Gamma_{j}^{\,\prime}\,P_{\,W^{\,\prime}_{j}}\,(\,g\,) \;,\; f_{\,1} \,\otimes\, g_{\,1}\,\right>\hspace{2.7cm}\]
\[=\, \left(\,\sum\limits_{\,i \,\in\, I}\,v_{\,i}\,v^{\,\prime}_{\,i}\,\left<\,P_{\,V_{i}}\,\Lambda_{i}^{\,\ast}\,\Lambda_{i}^{\,\prime}\,P_{\,V^{\,\prime}_{i}}\,\,(\,f\,) \,,\, f_{\,1}\,\right>_{\,1}\,\right)\,\left(\,\sum\limits_{\,j \,\in\, J}\,w_{\,j}\,w^{\,\prime}_{\,j}\,\left<\,P_{\,W_{j}}\,\Gamma_{j}^{\,\ast}\,\Gamma_{j}^{\,\prime}
\,P_{\,W^{\,\prime}_{j}}\,(\,g\,) \,,\, g_{\,1}\,\right>_{\,2}\,\right).\]
Then by Cauchy-Schwartz inequality, 
\[\left|\,\left<\,S\,(\,f \,\otimes\, g\,) \;,\; f_{\,1} \,\otimes\, g_{\,1}\,\right>\,\right|\]
\[\,\leq\, \left(\,\sum\limits_{\,i \,\in\, I}\,\,(\,v^{\,\prime}_{\,i}\,)^{\,2}\,\left\|\,\Lambda_{i}^{\,\prime}\,P_{\,V^{\,\prime}_{i}}\,\,(\,f\,)\,\right\|_{\,1}^{\,2}\,\right)^{\,\dfrac{1}{2}}\,\left(\,\sum\limits_{\,i \,\in\, I}\,\,v_{\,i}^{\,2}\,\left\|\,\Lambda_{i}\,P_{\,V_{i}}\,\,(\,f_{\,1}\,)\,\right\|_{\,1}^{\,2}\,\right)^{\,\dfrac{1}{2}}\,\times\hspace{3cm}\]
\[\hspace{2.3cm}\left(\,\sum\limits_{\,j \,\in\, J}\,(\,w^{\,\prime}_{\,j}\,)^{\,2}\,\left\|\,\Gamma_{j}^{\,\prime}\,P_{\,W^{\,\prime}_{j}}\,\,(\,g\,)\,\right\|_{\,2}^{\,2}\,\right)^{\,\dfrac{1}{2}}\,\left(\,\sum\limits_{\,j \,\in\, J}\,w_{\,j}^{\,2}\,\left\|\,\Gamma_{j}\,P_{\,W_{j}}\,\,(\,g_{\,1}\,)\,\right\|_{\,2}^{\,2}\,\right)^{\,\dfrac{1}{2}}\]
\[=\, \left(\,\sum\limits_{\,i,\, j }\,(\,v^{\,\prime}_{i}\,)^{\,2}\,(\,w^{\,\prime}_{j}\,)^{\,2}\,\left\|\,\Lambda_{i}^{\,\prime}\,P_{\,V^{\,\prime}_{i}}\,\,(\,f\,)\,\right\|_{1}^{\,2}\,\left\|\,\Gamma_{j}^{\,\prime}\,P_{\,W^{\,\prime}_{j}}\,(\,g\,)\,\right\|_{2}^{\,2}\,\right)^{\,\dfrac{1}{2}}\,\times\hspace{3.3cm}\]
\[\,\left(\,\sum\limits_{\,i,\, j }\,v_{i}^{\,2}\,w_{j}^{\,2}\,\left\|\,\Lambda_{i}\,P_{\,V_{i}}\,(\,f_{\,1}\,)\,\right\|_{1}^{\,2}\,\left\|\,\Gamma_{j}\,P_{\,W_{j}}\,(\,g_{\,1}\,)\,\right\|_{2}^{\,2}\,\right)^{\,\dfrac{1}{2}}\]
\[=\, \left(\,\sum\limits_{i,\, j}\,(\,v^{\,\prime}_{\,i}\,)^{\,2}\,(\,w^{\,\prime}_{j}\,)^{\,2}\,\left\|\,\left(\,\Lambda^{\,\prime}_{i} \,\otimes\, \Gamma^{\,\prime}_{j}\,\right)\,P_{\,V^{\,\prime}_{i} \,\otimes\, W^{\,\prime}_{j}}\,(\,f \,\otimes\, g\,)\,\right\|^{\,2}\,\right)^{\,\dfrac{1}{2}} \,\times\hspace{3.2cm}\]
\[\hspace{2.7cm}\,\left(\,\sum\limits_{i,\, j}\,v^{\,2}_{\,i}\,w^{\,2}_{\,j}\,\left\|\,\left(\,\Lambda_{i} \,\otimes\, \Gamma_{j}\,\right)\,P_{\,V_{i} \,\otimes\, W_{j}}\,(\,f_{\,1} \,\otimes\, g_{\,1}\,)\,\right\|^{\,2}\,\right)^{\,\dfrac{1}{2}}\]
\[\;[\;\text{using (\ref{eq1.001}), (\ref{eq1.0001}) and applying the Theorem (\ref{th1.1})}\;]\]
\[\leq\, \sqrt{B_{\,1}\,B_{\,2}}\;\left\|\,f \,\otimes\, g\,\right\|\,\left\|\,f_{\,1} \,\otimes\, g_{\,1}\,\right\|.\hspace{8cm}\]
Let \,$S_{\Lambda\,\Lambda^{\,\prime}}$\, and \,$S_{\Gamma\,\Gamma^{\,\prime}}$\, be the frame operators for the pair of \,$g$-fusion Bessel sequences \,$(\,\Lambda,\, \Lambda^{\,\prime}\,)$\, and \,$(\,\Gamma,\, \Gamma^{\,\prime}\,)$, respectively.\;Then by above calculation
\[\left|\,\left<\,\left(\,S_{\Lambda\,\Lambda^{\,\prime}} \,\otimes\, S_{\Gamma\,\Gamma^{\,\prime}}\,\right)\,(\,f \,\otimes\, g\,) \;,\; f_{\,1} \,\otimes\, g_{\,1}\,\right>\,\right| \,\leq\, \sqrt{B_{\,1}\,B_{\,2}}\;\left\|\,f \,\otimes\, g\,\right\|\,\left\|\,f_{\,1} \,\otimes\, g_{\,1}\,\right\|\]
\[\Rightarrow\, \left|\,\left<\,S_{\Lambda\,\Lambda^{\,\prime}}\,(\,f\,) \,,\, f_{\,1}\,\right>_{1}\,\right|\, \left|\,\left<\,S_{\Gamma\,\Gamma^{\,\prime}}\,(\,g\,) \,,\, g_{\,1}\,\right>_{2}\,\right| \,\leq\, \sqrt{B_{\,1}\,B_{\,2}}\;\|\,f\,\|_{1}\;\|\,f_{\,1}\,\|_{1}\;\|\,g\,\|_{2}\;\|\,g_{\,1}\,\|_{2}\hspace{1cm}\]
\[\Rightarrow\, \sup\limits_{\,\|\,f_{\,1}\,\|_{1} \,=\, 1}\;\left|\,\left<\,S_{\Lambda\,\Lambda^{\,\prime}}\,(\,f\,) \,,\, f_{\,1}\,\right>_{1}\,\right|\,\sup\limits_{\,\|\,g_{\,1}\,\|_{2} \,=\, 1}\;\left|\,\left<\,S_{\Gamma\,\Gamma^{\,\prime}}\,(\,g\,) \,,\, g_{\,1}\,\right>_{2}\,\right| \,\leq\, \sqrt{B_{\,1}\,B_{\,2}}\;\|\,f\,\|_{1}\;\|\,g\,\|_{2}\]
\[\Rightarrow\, \left\|\,S_{\Lambda\,\Lambda^{\,\prime}}\,(\,f\,)\,\right\|_{1}\;\left\|\,S_{\Gamma\,\Gamma^{\,\prime}}\,(\,g\,)\,\right\|_{2} \,\leq\, \sqrt{B_{\,1}\,B_{\,2}}\;\|\,f\,\|_{1}\;\|\,g\,\|_{2}\]
\[\Rightarrow\, \dfrac{\left\|\,S_{\Lambda\,\Lambda^{\,\prime}}\,(\,f\,)\,\right\|_{1}}{\|\,f\,\|_{1}}\;\dfrac{\left\|\,S_{\Gamma\,\Gamma^{\,\prime}}\,(\,g\,)\,\right\|_{2}}{\|\,g\,\|_{2}} \,\leq\, \sqrt{B_{\,1}\,B_{\,2}}.\hspace{1.8cm}\]
Again taking supremum on both side with respect to \,$\|\,f\,\|_{\,1} \,=\, 1$\, and \,$\|\,g\,\|_{\,2} \,=\, 1$,
\[\|\,S\,\| \,=\, \left\|\,S_{\Lambda\,\Lambda^{\,\prime}} \,\otimes\, S_{\Gamma\,\Gamma^{\,\prime}}\,\right\| \,=\, \left\|\,S_{\Lambda\,\Lambda^{\,\prime}}\,\right\|\,\left\|\,S_{\Gamma\,\Gamma^{\,\prime}}\,\right\| \,\leq\, \sqrt{B_{\,1}\,B_{\,2}}.\]
This completes the proof.  
\end{proof}

\end{document}